\documentclass[a4paper,12pt,leqno]{amsart}
\usepackage{amsmath,amssymb,array,euscript,amsfonts}
\usepackage{graphicx,epsfig,color}
\usepackage[dvipsnames]{xcolor}
\usepackage{rotating}
\usepackage{enumerate}
\usepackage{bbm}
\usepackage{color}
\usepackage{hyperref}
\usepackage{lmodern,microtype}
\usepackage{geometry}
\usepackage[yyyymmdd,hhmmss]{datetime}
\usepackage[numbers,sort,compress]{natbib}
\usepackage{amstext,amsthm,amssymb}
\usepackage{xparse,mathrsfs,mathtools}
\usepackage[british]{babel}
\usepackage[babel=true,english=british]{csquotes}
\usepackage{xr-hyper}
\usepackage{stmaryrd}
\usepackage[capitalise,nameinlink]{cleveref}
\usepackage[cal=zapfc,bb=ams,frak=esstix,scr=boondoxo]{mathalfa}

\def\N{\hbox{\font\dubl=msbm10 scaled 1000 {\dubl N}}}

\DeclarePairedDelimiter\parentheses{\lparen}{\rparen}
\DeclarePairedDelimiter\abs{\lvert}{\rvert}
\DeclarePairedDelimiter\brackets{\lbrack}{\rbrack}
\DeclarePairedDelimiter\braces{\lbrace}{\rbrace}
\def\d{\,{\rm{d}}}

\newtheorem{Theorem}{Theorem}

\newtheorem{Cor}{Corollary}

\newtheorem{Lem}{Lemma}
\title[The Lindel\"of Hypothesis for Zeta Zero Ordinates]{The Lindel\"of Hypothesis for Zeta Zero Ordinates}

\thanks{\bf Dedicated to the memory of Aleksandar Ivi\'c}

\author[Ram\={u}nas Garunk\v{s}tis, Athanasios Sourmelidis, J\"orn Steuding]{Ram\={u}nas Garunk\v{s}tis, Athanasios Sourmelidis, J\"orn Steuding}

\begin{document}
	
	\begin{abstract}
		We provide conditional and unconditional asymptotic formulae for the exponential sums $\sum_\gamma\,\gamma^{-i\tau}$, where the summation is over the ordinates of the nontrivial zeros $\rho=\beta+i\gamma$ of the Riemann zeta-function.
		 In particular, the obtained results are related to the Lindel\"of Hypothesis for these ordinates (in the sense of Gonek et al. \cite{gonek}).    
	\end{abstract}
	
	\maketitle
	
	{\small \noindent {\sc Keywords:} Riemann zeta-function, $S(t)$, nontrivial zeros, Lindel\"of Hypothesis, Riemann Hypothesis}
	\bigskip
	
	\section{Motivation and Statement of the Main Results} 
	
	The distribution of the zeros of the Riemann zeta-function $\zeta$ fascinates and occupies generations of mathematicians. When ambiguities exist, new perspectives in particular are valuable. This article discusses such a new perspective, recently introduced by Gonek, Graham \& Lee \cite{gonek}.
	
	Let $\rho=\beta+i\gamma$ denote the nontrivial (non-real) zeros of $\zeta(s)$. Following Gonek et al. \cite{gonek}, the Lindel\"of Hypothesis for the sequence of ordinates of the nontrivial zeros of $\zeta$ is the statement that, for real $x$ and $\tau$, the asymptotic formula
	\begin{equation}\label{LHforZEROS}
		\sum_{0<\gamma<x}\gamma^{-i\tau}=\int_2^x\frac{u^{-i\tau}}{2\pi}\log\frac{u}{2\pi}\d u+{\mathcal E}
	\end{equation}
	holds with an error term ${\mathcal E}$ satisfying
	\begin{equation}\label{error}
		{\mathcal E}\ll \vert \tau\vert^\epsilon x^{1/2}
	\end{equation}
	for $2\leq x\leq \vert \tau\vert^B$ and any $B>0$\footnote{All implicit constants may depend only on the arbitrary but fixed $\epsilon>0$.}; we abbreviate this by $\mathrm{LH}(\gamma)$. Note that the lower limit of the integral $u\geq 2$ is chosen to be strictly smaller than the minimum of the positive ordinates $\gamma$ (which is slightly larger than $14$). The asymptotic formula \eqref{LHforZEROS} with the error term \eqref{error} is trivially valid for $2\le x \le |\tau|^\epsilon$.

	Our first approach to prove \eqref{LHforZEROS} with the error term \eqref{error} relies on the argument of $\zeta$. This function is defined in a standard way (see \cite[Section 9.3]{tit2})  by 
	\begin{equation*}
		S(T)={\textstyle{\frac{1}{\pi}}}\arg\zeta({\textstyle{\frac{1}{2}}}+iT)
	\end{equation*}
	and appears in the Riemann-von Mangoldt formula for the number $N(T)$ of nontrivial zeros $\rho=\beta+i\gamma$ satisfying $0<\gamma<T$ (counting multiplicities), i.e.
	\begin{equation}\label{RvM}
		N(T)=\frac{T}{2\pi}\log\frac{T}{2\pi e}+{\textstyle{\frac{7}{8}}}+S(T)+f(T),
	\end{equation}
	where $f(T)\ll T^{-1}$ is an odd power series in $T^{-1}$. The present best estimate for $S$ had already been given by von Mangoldt \cite{mangoldt}, who showed that
	\begin{equation}\label{litt}
		S(T)\ll \log T .
	\end{equation}
	A simplified proof of this estimate is due to Landau \cite{landau}; his approach can still be found in modern literature. Littlewood \cite{littlewood} replaced the right hand side in (\ref{litt}) by $\log T/\log\log T$ under assumption of the Riemann Hypothesis, and there is still no better conditional estimate of $S$ known. 
	
	Estimate (\ref{litt}) not only allows an asymptotic Riemann-von Mangoldt formula with a good error term for (\ref{RvM}) but provides the following  result towards $\mathrm{LH}(\gamma)$.
	
	\begin{Theorem}\label{3}
		For any real numbers $\tau$ and $x>0$ we have that
		$$
		\sum_{0<\gamma<x}\gamma^{-i\tau}=\frac{x^{1-i\tau}}{2\pi(1-i\tau)}\left(\log \frac{x}{2\pi}-\frac{1}{1-i\tau}\right)+O(|\tau|(\log x)^2). 
		$$
	\end{Theorem}

It follows immediately from Theorem \ref{3} that the $\mathrm{LH}(\gamma)$ is true if $|\tau|^2\leq x\leq \tau^B$ for arbitrary fixed $B\geq2$.
	The main term is equal to the integral in (\ref{LHforZEROS}) up to a bounded quantity (as follows from the Riemann-von Mangoldt formula with error term, see (\ref{belo}) below). 
	Note that it is essentially this range of small $\vert\tau\vert\ll\sqrt{x}$ for which the main term dominates the error term ${\mathcal E}$ in (\ref{LHforZEROS}). 
	Recall also that the classical Lindel\"of Hypothesis 
	\begin{align}\label{lh}
	\zeta\left(\frac{1}{2}+it\right)\ll|t|^\epsilon,\quad|t|\geq1,
	\end{align}
	is equivalent to the estimate 
\[
\sum_{n\leq x}n^{i\tau}\ll \vert \tau\vert^\epsilon x^{1/2},\quad1\leq x\leq \vert \tau\vert^{1/2},
\]
 as it is shown in Karatsuba and Voronin \cite[p.162]{kavo}. 
This indicates as well that small values for $x$ (or large values for $\vert \tau\vert$) are crucial. 

To prove the conjecture to the full range seems to be difficult.
As the proof of Theorem \ref{3} shows, it would be required to bound non-trivially the term $i\tau\int_{0}^{x}S(t)t^{-1-i\tau}\mathrm{d}\tau$.
Regardless, if we were to succeed, then the error term dominates the main term and the sum exhibits square root cancellation which would suffice and is often the optimal result one may hope to obtain.
This line of research has been pursued already by Fujii \cite{fujii} where he estimated $\int_{0}^{x}f'(t)S(t)e^{if(t)}\mathrm{d}t$ for a certain class of functions.
His approach is based on Littlewood's lemma which is an integrated version of the residue theorem.
In our case $f(t)=-\tau\log t$ is not included in that class but only for a small technical detail.
Fujii also fixes $\tau$, while we wish to have it run uniformly in certain ranges of powers of $x$.
It  is probable that his approach can be applied in this case as well.

An alternative approach is to study the sum
	\[
\mathcal{I}(x,y,\tau):=\sum_{x\leq\gamma< y}\parentheses*{i\parentheses*{\frac{1}{2}-\rho}}^{-i\tau},\quad 0<x<y,\,\tau\in\mathbb{R},
\]  
unconditionally first and then assume RH. 
To that end we will prove  the next theorem by appealing directly to the residue theorem.
\begin{Theorem}\label{unconditional1}
	Let $c\in(0,1)$.
	Uniformly on $2\leq x\ll y\ll x$ and $0\leq\tau\ll x\log x$
	\[
	\mathcal{I}(x,y,\tau)=\int_x^{y}\frac{t^{-i\tau}}{2\pi}\log\frac{t}{2\pi}\mathrm{d}t+O\parentheses*{(\log x)^2},\quad0\leq\tau\leq cx\log 2,
	\]
	and
	\[
			\mathcal{I}(x,y,\tau)=\frac{e^{i(\tau+\pi /4)}{\tau}^{1/2-i\tau}}{\parentheses{2\pi}^{1/2}}\sum_{e^{\tau/y}\leq n\leq e^{\tau/x}}\frac{\Lambda(n)(\log n)^{i\tau}}{{n}^{1/2}\log n}+\mathcal{F},\quad cx\log 2\leq\tau\ll x\log x.
		\]
		Here $\Lambda(n)$ is the von Mangoldt function and
		\[
		\mathcal{F}\ll e^{\tau/(2x)}\parentheses*{(\log x)^2+\parentheses*{\frac{\tau}{x}}^2\log x}+\tau^{1/2}e^{\tau/(2x)-\tau/y}\frac{\tau}{x}.
	\]
\end{Theorem}
The above result agrees with \cite[Theorem]{fujii} when $\tau$ is fixed.
Applying the trivial bound on the exponential sum involving $\Lambda(n)$ readily yields the following conditional result.

\begin{Cor}\label{cond}
On the Riemann hypothesis we have that
\[
\sum_{{x\leq\gamma<2x}}\gamma^{-i\tau}\ll \frac{x\log x}{|\tau|}+(\log x)^2,\quad1\leq|\tau|\leq cx\log 2,
\]
and
\[
\sum_{{x\leq\gamma<2x}}\gamma^{-i\tau}=x^{1/2+o(1)},\quad cx\log2\leq|\tau|=o(x\log x).
\]
\end{Cor}

We next draw a connection on the conjectural behavior of the sum in \eqref{LHforZEROS} for the remaining range $2\leq x\leq|\tau|^2$ with the order of growth of the Dirichlet series
\[
G(s):=\sum_{\gamma>0}\frac{1}{\gamma^s},\quad\sigma>1,
\]
which was introduced by Delsarte \cite{delsarte}, and recent work of Ivi\'c \cite{ivic} and Bondarenko et al. \cite{bonda}.  
It follows from the Riemann-von Mangoldt formula (\ref{RvM}) \& (\ref{litt}) as well as $f(T)\ll T^{-1}$ that the $n$th ordinate of a nontrivial zero (in ascending order and in the upper half-plane) is asymptotically equal to $\gamma_n \sim 2\pi n/ \log n$. Hence, $G(s)$ converges in the right half-plane $\sigma>1$ absolutely and uniformly in any compact subset and, therefore, it is an analytic function for $\sigma>1$. 
Summation by parts shows that it can be continued meromorphically to the half-plane $\sigma>-1$ with a double pole at $s=1$.
By a straightforward approach (relying on integration by parts in combination with (\ref{RvM}) \& (\ref{litt})), Ivi\'c \cite{ivic} achieved the estimate 
\[
G(\sigma+it)\ll |t|^{1-\sigma}\log|t|+(\log|t|)^2,\quad0\leq\sigma\leq1,\quad|t|\geq t_0>0.
\] 

We show that better estimates  on the vertical line $1/2+i\mathbb{R}$ for $G(s)$ entail information on the $\mathrm{LH}(\gamma)$ and vice versa.
\begin{Theorem}\label{order}
	Let $A\in[0,1/2]$ and $\epsilon>0$ be fixed.
	If $\mathcal{E}$ is the error term in \eqref{LHforZEROS} then the following are equivalent:
	\begin{enumerate}[i)]
		\item $	\mathcal{E}\ll|\tau|^{A+\epsilon}{x}^{1/2}$ for any $2\leq x\leq|\tau|^2$,
		\item $G(1/2+it)\ll |t|^{A+\epsilon}$ for any $|t|\geq1$.
	\end{enumerate}
\end{Theorem}

\paragraph{\bf Structure of the paper} In the next two sections, we prove Theorems 1 and 2 correspondingly. Sections 4--6 are devoted to the proof of Theorem 3. 
In the last section, we give several historical remarks.
\section{Proof of Theorem \texorpdfstring{\ref{3}}{1}: Stieltjes Integration}

We begin with the integral in (\ref{LHforZEROS}) and show that it is equal to the explicit main term in the theorem.
In fact, writing $n(u):=\frac{u}{2\pi}\log\frac{u}{2\pi e}$ for the main term in the Riemann-von Mangoldt formula (\ref{RvM}), it follows that 
\begin{align}\label{belo}
	\begin{split}
	{\mathcal I}&:=\int_2^x\frac{u^{-i\tau}}{2\pi}\log\frac{u}{2\pi}\d u =\int_2^xn'(u)\,u^{-i\tau}\d u\\
	&= n(u)\,u^{-i\tau}\Big\vert_{u=2}^x+i\tau\int_2^x n(u)\,u^{-i\tau-1}\d u \\
	&= \frac{x^{1-i\tau}}{2\pi}\log\frac{x}{2\pi e}+O(1)+i\tau\left({\mathcal I}-\int_2^x\frac{u^{-i\tau}}{2\pi}\d u\right).
	\end{split}
\end{align} 
Computing the latter integral and solving the equation (\ref{belo}) for ${\mathcal I}$ yields that
\begin{equation}\label{ee}
	{\mathcal I}=\frac{x^{1-i\tau}}{2\pi(1-i\tau)}\left(\log\frac{x}{2\pi}-\frac{1}{1-i\tau}\right)+O(1).
\end{equation}
On the other hand,  Stieltjes integration implies that
\[
\sum_{0<\gamma<x}\gamma^{-i\tau}=\int_2^x u^{-i\tau}\d N(u)+O(1).
\]
From the Riemann-von Mangoldt formula (\ref{RvM}) it follows that
$$
\sum_{0<\gamma<x}\gamma^{-i\tau}=\mathcal{I}+\int_2^x u^{-i\tau}\d \big(S(u)+f(u)\big)+O(1).
$$
The contribution of $f$ is negligible here. In view of $f'(u)\ll u^{-2}$ (see \cite{tit2}, proof of Theorem 14.14 (A)) we have
$$
\int_2^xu^{-i\tau}\d f(u)\ll 1.
$$
Hence, we get from \eqref{ee} that
$$
{\mathcal E}=\sum_{0<\gamma<x}\gamma^{-i\tau}-\frac{x^{1-i\tau}}{2\pi(1-i\tau)}\left(\log\frac{x}{2\pi}-\frac{1}{1-i\tau}\right)=\int_2^x u^{-i\tau}\d S(u)+O(1).
$$
Integrating by parts shows for the integral on the right that
\begin{equation}\label{sozi}
	\int_2^x u^{-i\tau}\d\,S(u)=S(u)u^{-i\tau}\Big\vert_{u=2}^x+i\tau\int_2^x S(u)u^{-1-i\tau}\d u.
\end{equation}
Applying the bound (\ref{litt}) in (\ref{sozi}), leads to 
$$
{\mathcal E}\ll|\tau|(\log x)^2.
$$

\section{Proof of Theorem \texorpdfstring{\ref{unconditional1}}{2}: Contour Integration} 
For the proof of the theorem we will need to employ the first and second derivative test and the stationary phase method. 
The following lemma which can be found in \cite[\S, Lemma 2]{kavo} collects all of these techniques.
Its last statement is implicitly given therein.
Here and in the sequel $\mathrm{e}(x):=e^{2\pi i x}$, $x\in\mathbb{R}$.
\begin{Lem}\label{stationary phase}
	Let $1\leq x<y\ll x$ and $f\in C^4[x,y]$ and $\phi\in C^2[x,y]$ be real-valued functions such that for some $U\geq y-x$, $0<A\leq U$ and $H>0$
	\begin{align*}
		\begin{array}{lll}
			\abs*{f''(t)}\asymp {A}^{-1},& f^3(t)\ll \parentheses{UA}^{-1},& f^4(t)\ll \parentheses*{U^2A}^{-1}
			\\		\phi(t)\ll H,& \phi'(t)\ll {H}{U}^{-1},& \phi''(t)\ll {H}{U}^{-2}
		\end{array}
	\end{align*}
	uniformly on $t\in[x,y]$.
	If $f'(t_0)=0$ for some $t_0\in[x,y]$ and $f''(t)>0$, then
	\begin{align*}
		\int_{x}^{y}\phi(t)\mathrm{e}(f(t))\mathrm{d}t
		&=\frac{e^{\pi i/4}\phi(t_0)\mathrm{e}(f(t_0))}{\parentheses*{f''(t_0)}^{1/2}}+O\parentheses*{H{E}},
	\end{align*}
	where 
	\[
	{E}:={A}{U}^{-1}+\min\parentheses*{{A}^{1/2},\abs*{f'(x)}^{-1}}+\min\parentheses*{{A}^{1/2},\abs*{f'(y)}^{-1}}.
	\]
	If no such $t_0\in[x,y]$ exists, then the integral is simply bounded by $H{E}$.
\end{Lem}
	\begin{proof}[Proof of Theorem \ref{unconditional1}]
	Let us assume first that $x$ and $y$ are not ordinates of zeta-zeros.
	The logarithmic derivative $\zeta'/\zeta$ has simple poles at the zeros of $\zeta$, and is regular elsewhere except for a simple pole at $s=1$. 
	Since $\parentheses*{i\parentheses*{\frac{1}{2}-s}}^{-i\tau}$ is analytic in the upper half-plane, Cauchy's theorem yields that
	\[
	\mathcal{I}(x,y,\tau)=\frac{1}{2\pi i}\int_{\square}\frac{\zeta'}{\zeta}(s)\parentheses*{i\parentheses*{\frac{1}{ 2}-s}}^{-i\tau}\mathrm{d} s,
	\]
	where the integration is over a positively oriented rectangular contour $\square$ with vertices $1/2-\delta+ix,{3}/{2}+ix, {3}/{2}+iy,1/2-\delta+iy$.		
	Here $\eta:=1/\log x$ and $\delta:=1/2+\eta$.
	We decompose the contour integral into four integrals 
	\begin{align}\label{beck}
		\sum_{j\leq 4}I_j:=\frac{1}{2\pi i}\braces*{\int_{-\eta+ix}^{3/2+ix}+\int_{3/2+ix}^{3/2+iy}+\int_{3/2+iy}^{-\eta+iy}+\int_{-\eta+iy}^{-\eta+ix} }\frac{\zeta'}{ \zeta}(s)\parentheses*{i\parentheses*{\frac{1}{2}-s}}^{-i\tau}\mathrm{d} s
	\end{align}
	
	Concerning the exponential factor, we observe that for any $t\geq {\tau}^{1/2}$
	\begin{align}\label{exponentialfactor}
		\begin{split}
			\parentheses*{i\parentheses*{\frac{1}{2}-s}}^{-i\tau}
			&=\exp\parentheses*{\tau\arg\parentheses*{t+i\parentheses*{\frac{1}{2}-\sigma}}}\exp\parentheses*{-i\tau\log\abs*{t+i\parentheses*{\frac{1}{2}-\sigma}}}\\
			&=e^{(1/2-\sigma)\tau/t}\mathrm{e}\parentheses*{-\frac{\tau\log t}{2\pi}}\parentheses*{1+O\parentheses*{\frac{\tau}{t^2}}}.
		\end{split}
	\end{align}
	
	We begin with the horizontal integrals in (\ref{beck}). 
	For this purpose we employ the classical bound \cite[(12.20)]{ivicbook}
	\begin{align}\label{logz}
		\frac{\zeta'}{\zeta}(s)\ll(\log t)^2,\quad t \geq1,\,-1\leq\sigma\leq2.
	\end{align}
	We then have that
	\begin{align}\label{i1i3}
	I_1\ll(\log x)^2\int_{-\eta}^{3/2}e^{(1/2-\sigma)\tau/x}\mathrm{d}\sigma\ll	e^{\tau/(2x)}(\log x)^2.
	\end{align}
	A similar computation and $y\ll x$ shows that $I_3$ can be estimated as $I_1$.
	
	For the vertical integral $I_2$ we employ the absolutely convergent Dirichlet series representation of $\zeta'/\zeta$ in the half-plane $\sigma>1$ and interchange summation and integration to obtain that
	\[
	\overline{I_2}=-\frac{1}{2\pi}\sum_{n\geq2}\frac{\Lambda(n)}{n^{3/2}}\int_{x}^{y}e^{-\tau/t}\mathrm{e}\parentheses*{\frac{\tau\log t+t\log n}{2\pi}}\mathrm{d}t+O\parentheses*{\frac{\tau}{x}\sum_{n\geq2}\frac{\Lambda(n)}{n^{3/2}}}.
	\]
	Since the first factor of the integrand is a monotonically increasing positive function for $t\in[x,y]$ and
	\[
	\frac{\mathrm{d}}{\mathrm{d}t}\brackets*{\tau\log t+t\log n}=\frac{\tau }{t}+\log n\geq\log2,\quad t\in[x,y],
	\]
	is also monotonic for $t\in[x,y]$, the last statement of
	Lemma \ref{stationary phase} implies that
	\begin{equation}\label{I2}
		I_2\ll\sum_{n\geq2}\frac{\Lambda(n)}{n^{3/2}}e^{-\tau/y}+\frac{\tau}{x}\ll1+\frac{\tau}{x}.
	\end{equation}
	
	The case of the vertical integral $I_4$ can become more delicate. 
	Here we recall the functional equation \cite[(1.24)]{ivicbook} of the zeta-function in the form
	\[
	\zeta(s)=\Delta(s)\zeta(1-s),\quad \Delta(s)=2^s\pi^{s-1}\sin\parentheses*{\frac{\pi s}{2}}\Gamma(1-s),\quad s\in\mathbb{C},
	\]
	which implies that
	\[
	\frac{\zeta'}{\zeta}(s)=\frac{\Delta'}{\Delta}(s)-\frac{\zeta'}{\zeta}(1-s),\quad s\in\mathbb{C}.
	\]
	Inserting this in $I_4$, we find that
	\begin{align}\label{david}
		I_4
		=\frac{1}{2\pi i}\int_{-\eta+iy}^{-\eta+ix}\parentheses*{\frac{\Delta'}{\Delta}(s)-\frac{\zeta'}{\zeta}(1-s)}\parentheses*{i\parentheses*{\frac{1}{2}-s}}^{-i\tau}\mathrm{d} s=: J_1+J_2.
	\end{align}
	
	The integral term of $\mathcal{I}(x,y,\tau)$ arises from the integral $J_1$ over the logarithmic derivative of $\Delta$. 
	By definition $\Delta'/\Delta$ is a meromorphic function whose (simple) poles lie on the real line while Stirling's formula (see \cite[(1.25)]{ivicbook}) implies that
	\begin{equation*}
		\frac{\Delta'}{\Delta}(s)=-\log\frac{\abs*{t}}{2\pi}+O\parentheses*{\frac{1}{ |t|}}\ll\log|t|,\quad|t|\geq2,\,|\sigma|\leq 1.
	\end{equation*}
	By Cauchy's theorem we see that
	\[
	-J_1=\frac{1}{2\pi i}\braces*{\int_{-\eta+ix}^{1/2+ix}+\int_{1/2+ix}^{1/2+iy}+\int_{1/2+iy}^{-\eta+iy}}\frac{\Delta'}{\Delta}(s)\parentheses*{i\parentheses*{\frac{1}{2}-s}}^{-i\tau}\mathrm{d}s.
	\]
	The horizontal integrals can be estimated as $I_1$ and $I_3$, where this time we employ the bound of $\Delta'/\Delta$ instead of the one \eqref{logz} for $\zeta'/\zeta$, and they can be seen to be $O(e^{\tau/(2x)}\log x)$.
	As for the vertical integral we have in view of \eqref{exponentialfactor} and the asymptotic expansion of $\Delta'/\Delta$ that
	\begin{align}\label{J_1}
		J_1+O(e^{\tau/(2x)}\log x)=\int_x^{y}\frac{t^{-i\tau}}{2\pi}\log\frac{t}{2\pi}\mathrm{d}t+O\parentheses*{\int_x^{y}\frac{\mathrm{d}t}{t}}\ll\frac{x\log x}{\tau}.
	\end{align}

To compute the integral $J_2$ we employ, similar as with $I_2$, the absolutely convergent Dirichlet series representation of $\zeta'/\zeta$ to see that
	\begin{align}\label{J_2}
		\begin{split}
		J_2
		&=\frac{1}{2\pi}\sum_{n\geq2}\frac{\Lambda(n)}{n^{1+\eta}}\int_x^{y}e^{\delta\tau/t}\mathrm{e}\parentheses*{\frac{t\log n-\tau\log t}{2\pi}}\mathrm{d}t+O\parentheses*{\frac{e^{\tau/(2x)}\tau}{x}\sum_{n\geq2}\frac{\Lambda(n)}{n^{1+\eta}}}\\
		&=:\frac{1}{2\pi}\sum_{n\geq2}\frac{\Lambda(n)\mathcal{L}(n)}{n^{1+\eta}}+O\parentheses*{\frac{e^{\tau/(2x)}\tau\log x}{x}}.
		\end{split}
	\end{align}
	
	For the integrals $\mathcal{L}(n)$ we will consider two separate cases $\tau\leq cx\log 2$ and $c x\log 2\leq\tau\ll x\log x$.
	In the first case $\tau\leq cx\log 2$, we have that
		\[
	\frac{\mathrm{d}}{\mathrm{d}t}\brackets*{\tau\log t-t\log n}=\frac{\tau }{t}-\log n\leq(c-1)\log2,\quad t\in[x,y],
	\]
	and, thus, Lemma \ref{stationary phase} implies that
	\begin{align*}
	J_2\ll\sum_{n\geq2}\frac{\Lambda(n)}{n^{1+\eta}}+\log x\ll\log x.
	\end{align*}
	The first statement of the theorem follows now from \eqref{beck}, \eqref{i1i3}--\eqref{J_1} and the above estimate for $J_2$, when $x$ and $y$ are not ordinates of zeta-zeros.
	
	In the second case $c x\log 2\leq\tau\ll x\log x$, we employ Lemma \ref{stationary phase} with
	\[
	2\pi f(t)=t\log n-\tau\log t,\quad\phi(t)=e^{\delta\tau/t},
	\]
	\[
	U\asymp x,\quad A=\frac{x^2}{\tau},\quad H=\frac{e^{\tau/(2x)}\tau}{x},\quad t_0=\frac{\tau}{\log n}.
	\]
	With the above notations the requirements of the lemma are satisfied where $2\pi f'(t)=\log n-\tau/t$, $2\pi f''(t)=\tau/t^2$ and $t_0\in[x,y]$ if, and only if, $e^{\tau/y}\leq n\leq e^{\tau/x}$.
	Therefore,
	\begin{align}\label{J_2secondcase}
		\begin{split}
		\frac{1}{2\pi}\sum_{n\geq2}\frac{\Lambda(n)\mathcal{L}(n)}{n^{1+\eta}}
		&=	\parentheses*{\frac{\tau}{2\pi}}^{1/2}\sum_{e^{\tau/y}\leq n\leq e^{\tau/x}}\frac{e^{i\pi /4}\Lambda(n)}{{n}^{1/2}\log n}\parentheses*{\frac{e\log n}{\tau}}^{i\tau}\\
		&\quad+O\parentheses*{\frac{e^{\tau/(2x)}\tau}{x}\sum_{\log n\in\cup_{i\leq 7}K_i}\frac{E_n\log n}{n}},
		\end{split}
	\end{align}
	where $E_n$ is the error term in Lemma \ref{stationary phase} and varies with $\log n$ ranging over the intervals, $K_1,\dots,K_7$ say,
	\begin{align*}
		\begin{array}{ccc}
		\brackets*{\frac{\tau}{y}-1,\frac{\tau}{y}-\frac{{\tau}^{1/2}}{y}},& \brackets*{\frac{\tau}{y}-\frac{\tau^{1/2}}{y},\frac{\tau}{y}+\frac{\tau^{1/2}}{y}},& \brackets*{\frac{\tau}{y}+\frac{\tau^{1/2}}{y},\frac{\tau}{y}+1},\\
		&\brackets*{\frac{\tau}{y}+1,\frac{\tau}{x}-1},&\\
		\brackets*{\frac{\tau}{x}-1,\frac{\tau}{x}-\frac{\tau^{1/2}}{x}},&	\brackets*{\frac{\tau}{x}-\frac{\tau^{1/2}}{x},\frac{\tau}{x}+\frac{\tau^{1/2}}{x}},&	\brackets*{\frac{\tau}{x}+\frac{\tau^{1/2}}{x},\frac{\tau}{x}+1},
		\end{array}
	\end{align*}
	respectively.
	
	Starting with $K_1$ we see that
	\begin{align*}
		\sum_{\log n\in K_1}\frac{\log n}{n}E_n&\ll\frac{x}{\tau}\sum_{\log n\in K_1}\frac{\log n}{n}+ \sum_{\log n\in K_1}\frac{\log n}{n\parentheses*{\frac{\tau}{y}-\log n}}\\
		&\ll 1+\frac{\tau}{y}\int_{e^{\tau/y-1}}^{e^{\tau/y-\tau^{1/2}/y}}\frac{\mathrm{d}u}{u\parentheses*{\frac{\tau}{y}-\log u}}+{e^{-\tau/y}}\tau^{1/2}\\
		&\ll\frac{\tau\log x}{x}+{e^{-\tau/y}}\tau^{1/2}.
	\end{align*}
	Completely analogously
	\begin{align*}
		\sum_{\log n\in K_3\cup K_5\cup K_7}\frac{\log n}{n}E_n\ll\frac{\tau\log x}{x}+{e^{-\tau/y}}\tau^{1/2}.
	\end{align*}
	Moving on the summation over the interval $K_2$ we have that
	\begin{align*}
		\sum_{\log n\in K_2}\frac{\log n}{n}E_n
		\ll\frac{x}{\tau}\sum_{\log n\in K_2}\frac{\log n}{n}+ e^{-\tau/y}\tau^{1/2}\sum_{\log n\in K_2}1\ll\frac{\tau}{x}+e^{-\tau/y}\tau^{1/2}
	\end{align*}
	and, similarly,
	\[
	\sum_{\log n\in K_6}\frac{\Lambda(n)}{n}E\ll\frac{\tau}{x}+e^{-\tau/y}\tau^{1/2}.
	\]
	Lastly,
	\[
	\sum_{\log n\in K_4}\frac{\log n}{n}E_n\ll\parentheses*{1+\frac{x}{\tau}}\sum_{\log n\in K_4}\frac{\log n}{n}\ll\frac{\tau}{x}.
	\]
	The second part of the theorem follows now from \eqref{beck}, \eqref{i1i3}--\eqref{J_2secondcase} and the above estimates, when $x$ and $y$ are not ordinates of zeta-zeros.
	
	Otherwise, the initial contour integration ought to be considered over a rectangle with vertices $-\frac{1}{2}+i(x-\epsilon),\frac{3}{2}+i(x-\epsilon), \frac{3}{2}+i(y-\epsilon),-\frac{1}{2}+i(y-\epsilon)$, where $\epsilon>0$ is sufficiently small such that $[x-\epsilon,x)\cup[y-\epsilon,y)$ does not contain any such ordinates.
	Then the proof remains unaltered.	
\end{proof}
\section{Proof of Theorem \texorpdfstring{\ref{order}: $i)\Rightarrow ii)$}{3: i)-->(ii)}}
Before we give the proof of the one direction Theorem \ref{order} (in the following section) we need to discuss the analytic properties of the function $G$.
In comparison with the Dirichlet series defining $\zeta$ (resp. its first derivative) one may expect that $G$ has a double pole at $s=1$. Indeed, it follows from Stieltjes integration that
\begin{eqnarray}\label{analyt}
	G(s)=\frac{1}{2\pi(s-1)^2}-\frac{\log 2\pi}{2\pi(s-1)}+C_1+s\int_1^\infty \big(S(u)+f(u)\big)u^{-s-1}\d u,
\end{eqnarray}   
valid for $\sigma>0$ (by applying (\ref{litt}) and $f(u)\ll u^{-1}$), where $C_1$ is an absolute constant. This analytic continuation was first shown by Chakravarty \cite{chakra}. 
Applying partial integration, Ivi\'c \cite{ivic} obtained  
\begin{align}\label{analytplus}
	G(s)&=\frac{1}{2\pi(s-1)^2}-\frac{\log 2\pi}{2\pi(s-1)}+C_1+\\
	&\quad+s\int_1^\infty f(u)u^{-s-1}\d u+s(s+1)\int_1^\infty \int_1^uS(v)\d v\, u^{-s-2}\d u,\nonumber
\end{align}   
which yields an analytic continuation to $\sigma>-1$ (by (\ref{litt})). Bondarenko et al. \cite{bonda} have shown that  $G(s)$ admits a meromorphic continuation to $\mathbb{C}$ with further simple poles at $s=1-2n$, $n\in\N$, provided that the Riemann Hypothesis is true.
This result was already found by Delsarte \cite{delsarte} (by different means). 

In our case, we have (initially for $\sigma > 1$, then by analytic continuation for $\sigma > 0$) that
\[
G(s)=\sum_{0<\gamma<X}\gamma^{-s}+R_X(s),
\]
where $X\geq1$ and
\begin{align*}
	R_X(s)&:=\frac{X^{1-s}}{2\pi(s-1)}\log\frac{X}{2\pi}+\frac{X^{1-s}}{2\pi(s-1)^2}-X^{-s}(S(X)+f(X))+\\
	&\quad +s\int_X^\infty u^{-s-1}(S(u)+f(u))\d u
\end{align*}
(see for example \cite[ (4.17)]{bonda})).
If we choose $s=1/2+it$ and $X=t^2$, $|t|\geq2$, then the bound $S(u)+f(u)\ll \log u$, $u\geq2$, implies that
\[
R_X(1/2+it)\ll\log|t|.
\]
Moreover, Stieltjes integration by parts and our assumptions yield that
\begin{align*}
	\sum_{0<\gamma<t^2}\frac{1}{\gamma^{1/2+it}}
	&=\frac{\sum_{0<\gamma<t^2}\gamma^{-it}}{|t|}+\frac{1}{2}\int_1^{t^2}\frac{\sum_{0<\gamma<u}\gamma^{-it}}{u^{3/2}}\d u\\
	&\ll|t|^{A+\epsilon}+\int_{1}^{t^2}\left(\frac{\log u}{|t|u^{1/2}}+\frac{|t|^{A+\epsilon}}{u}\right)\d u\\
	&\ll|t|^{A+\epsilon}.
\end{align*}
Therefore, 
\[
G(1/2+it)\ll|t|^{A+\epsilon},\quad|t|\geq1,
\]
and this concludes the one direction of the theorem.
\section{Perron's Formula for \texorpdfstring{$G$}{G}} 

Our aim is to find an asymptotic formula for the left hand side (\ref{LHforZEROS}), where $\tau$ is real. For this purpose we recall the well-known formula (see \cite[Lemma 2.1]{ivicbook})
\[
\frac{1}{2\pi i}\int_{c-iT}^{c+iT}\frac{y^s}{s}\d s=\delta(y)+O\left(y^c\min\{1, (T\vert \log y\vert)^{-1}\}\right)
\]
with  $\delta(y)=1$ for $y>1$, $\delta(1)=\frac{1}{2}$ and $\delta(y)=0$ for $0<y<1$.

Applying this to the $G$ defining Dirichlet series in the half-plane of absolute convergence, we find, for $c>1$, 
\begin{equation}\label{perron}
	\sum_{0<\gamma<x}\gamma^{-i\tau}=\frac{1}{2\pi i}\int_{c-iT}^{c+iT}G(s+i\tau)\frac{x^s}{s}\d s+{\rm{E}},
\end{equation}
where
\begin{equation}\label{perronerror}
	{\rm{E}}\ll x^c\sum_{\gamma>0}\gamma^{-c}\min\{1, (T\vert \log (x/\gamma)\vert)^{-1}\},
\end{equation}
and $x, T$ are large and $T>2x$; here we also suppose (to avoid technicalities) that $x\neq \gamma$. In order to estimate this sum we split the range of summation into $\vert \gamma-x\vert\leq \frac{x}{T}$, $\frac{x}{T}<\vert \gamma-x\vert\leq \frac{x}{2}$ and $|\gamma-x|>\frac{x}{2}$.

For the range $\vert \gamma-x\vert\leq \frac{x}{T}$, the Riemann-von Mangoldt formula (\ref{RvM}) yields that
\begin{align}\label{axl1}
	\begin{split}
		x^c\sum_{\substack{\gamma>0\\\vert \gamma-x\vert\leq x/T}}\gamma^{-c}\min\{1, (T\vert \log (x/\gamma)\vert)^{-1}\}
		\ll\sum_{\substack{\gamma>0\\\vert \gamma-x\vert\leq x/T}}1\ll\log x.
	\end{split}
\end{align}

For the second range $\frac{x}{T}<\vert \gamma-x\vert\le \frac{x}{2}$, we will employ the relation
$$
|\log(\gamma/x)|=\left|\log\left(1+\frac{\gamma-x}{x}\right)\right|\gg \frac{|\gamma-x|}{x}.
$$
To that end
\begin{align*}
	x^c&\sum_{\substack{\gamma>0\\x/T<\vert  \gamma-x\vert\le x/2}}\gamma^{-c}\min\{1, (T\vert \log (x/\gamma)\vert)^{-1}\}\\
	&\ll \frac{x}{T}\sum_{x/2\leq \gamma< x-x/T}\frac{1}{x-\gamma}+\frac{x}{T}\sum_{x+x/T<  \gamma\le 3x/2}\frac{1}{\gamma-x}\\
	&\ll\sum_{j\leq T/2}\frac{1}{j}\sum_{x-(j+1)x/T\leq\gamma\leq x-jx/T}1+\sum_{j\leq T/2}\frac{1}{j+1}\sum_{x+jx/T\leq\gamma\leq x+(j+1)x/T}1.
\end{align*}

Once more, an application of the Riemann-von Mangoldt formula yields that
\[
x^c\sum_{\substack{\gamma>0\\x/T<\vert  \gamma-x\vert\le x/2}}\gamma^{-c}\min\{1, (T\vert \log (x/\gamma)\vert)^{-1}\}\ll(\log x)\log T.
\]
Lastly, for the range $\vert \gamma-x\vert>\frac{x}{2}$, we have $\vert\log(\gamma/x)\vert^{-1}\ll 1$ and, thus, for all sufficiently large $T$,
\begin{eqnarray*}
	x^c\sum_{\substack{\gamma>0\\ \vert \gamma-x\vert> x}}\gamma^{-c}\min\{1, (T\vert \log (x/\gamma)\vert)^{-1}\}\ll \frac{x^c}{T}\sum_{\gamma>0}\gamma^{-c}\ll \frac{x^c}{T}G(c).
\end{eqnarray*}
In view of (\ref{analyt}) we may replace here $G(c)$ by $(c-1)^{-2}$.
Collecting all estimates we deduce that
\begin{align}\label{perron2}
	\begin{split}
		\sum_{0<\gamma<x}\gamma^{-i\tau}
		&=\frac{1}{2\pi i}\int_{c-iT}^{c+iT}G(s+i\tau)\frac{x^s}{s}\d s+O\left(\frac{x^c}{(c-1)^2T}+(\log x)\log T\right).
	\end{split}
\end{align}
We may now also allow $x=\gamma$ here at the expense of an additional error $x^{-i\tau}\ll 1$ which is negligible.

The idea of a Perron formula for general Dirichlet series can already be found in \cite{perro}; for our purpose, however, a truncated version (with finite $T$ is useful different to the infinite one ($T=\infty$) in the original paper or in all  other relevant papers the authors are aware of.

\section{Proof of Theorem \texorpdfstring{\ref{order}: $ii)\Rightarrow i)$}{3: ii)-->i)}}   

We replace the integral in (\ref{perron2}), $I_G$ say, by contour integration. Consider the rectangular contour with vertices $c\pm iT$ and $1/2\pm iT$, where $c:=1+\epsilon$ and $T=\tau^2$.
Then, by Cauchy's theorem,
\begin{align}\label{cont}
	I_G=\left\{\int_{c-iT}^{1/2-iT}+\int_{1/2-iT}^{1/2+iT}+\int_{1/2+iT}^{c+iT}\right\}G(s+i\tau)\frac{x^s}{s}\d s +2\pi i{\rm{res}},
\end{align}
where ${\rm{res}}$ denotes the only residue of the integrand inside the contour of integration corresponding to the double pole of $G(s+i\tau)$ at $s=1-i\tau$ (by (\ref{analyt})). In fact, taking into the Laurent expansion (\ref{analyt}) in combination with 
\[
\frac{x^s}{s}=\frac{x^\omega}{\omega}+\frac{x^\omega}{\omega}\left(\log x-\frac{1}{\omega}\right)(s-\omega)+O(\vert s-\omega\vert^2),
\]
specialized for $\omega=1-i\tau$, it follows that
\begin{equation}\label{residue}
	{\rm{res}}=\frac{x^{1-i\tau}}{2\pi(1-i\tau)}\left(\log \frac{x}{2\pi}-\frac{1}{1-i\tau}\right).
\end{equation}
Another integration by parts shows that the quantity on the right is up to a bounded term $O(1)$ equal to the integral in (\ref{LHforZEROS}) (as follows from (\ref{belo})). Hence, in view of (\ref{perron2}) and (\ref{cont}), we have found the main term of the desired asymptotic formula. It remains to bound the integrals over the line segments in (\ref{cont}).

For the integrals over the horizontal line segments, we find, by our assumption on the growth of $G(1/2+it)$ and the Phragm\'en-Lindel\"of principle, that
\begin{align*}
	\int_{1/2\pm iT}^{c\pm iT}G(s+i\tau)\frac{x ^s}{s}\d s\ll \int_{1/2}^{1+\epsilon}T^{A-1+\epsilon} x^\sigma\d \sigma\ll x^{1+\epsilon}\tau^{2A-2+\epsilon}\ll|\tau|^{A+\epsilon}{x}^{1/2}
\end{align*}
since $x\leq|\tau|^2$ and $A\leq\frac{1}{2}$.
For the vertical integral on the left we find that
\begin{align*}\label{ev}
	\int_{1/2-iT}^{1/2+iT}G(s+i\tau)\frac{x ^s}{s}\d s\ll \int_{-T}^{T}\frac{\sqrt{x}|\tau|^{A+\epsilon}}{\left|\frac{1}{2}+it\right|}\d t\ll|\tau|^{A+\epsilon}{x}^{1/2}.
\end{align*}
Substituting these bounds in (\ref{cont}) and taking into account \eqref{residue} and \eqref{perron2} we finally arrive at 
\begin{align*}
	{\mathcal E}&=\left\vert \sum_{0<\gamma<x}\gamma^{-i\tau}-\frac{x^{1-i\tau}}{2\pi(1-i\tau)}\left(\log \frac{x}{2\pi}-\frac{1}{1-i\tau}\right) \right\vert\\
	&\ll (\log x)\log T+|\tau|^{A+\epsilon}{x}^{1/2}\\
	&\ll |\tau|^{A+\epsilon}{x}^{1/2}.
\end{align*} 

\section{Concluding Historical Remarks} 

The classical Lindel\"of Hypothesis (\ref{lh}) dates back to a paper by Lindel\"of and Phragm\'en \cite{liph} (including their Phragm\'en-Lindel\"of principle) and a sequel due to Lindel\"of \cite{linde} in which he asks: ``D'ailleurs, le module $\vert\zeta(s)\vert$ ne resterait-il pas inf\'erieur \`a une limite finie pour $\sigma\geq 1/2+\epsilon, \vert t\vert>t_0(>0)$, quelque petit qu'on se donne le nombre positif $\epsilon$?'' In 1910, Bohr\footnote{The paper is authored by Bohr and Landau, however, in the introduction  it is explained that Bohr was the one proving this unboundedness.} \cite{boh} gave a negative answer to this question by showing that $\zeta(s)$ is unbounded to the right of the critical line, also distant to the simple pole. The naming  {\it Lindel\"of Hypothesis} for (\ref{lh}) is probably due to the influential works of Hardy \& Littlewood (in particular \cite{hl}) about a decade later.  

The recent paper \cite{gonek} of Gonek et al.  has put the classical Lindel\"of Hypothesis in an interesting new and rather general context. 
The first and third author of this article were discussing a Lindel\"of hypothesis for Lerch and Hurwitz zeta-functions in \cite{garunk} (also in a form close to that proposed by Gonek et al.), and Laurin\v cikas \& Macaitien\.e \cite{lama} continued this line of research. 
Conrey \& Ghosh \cite{cg} discussed the generalization of the Lindel\"of Hypothesis for $L$-functions (in the frame of the Selberg class); here is the case of fixed $\tau$ of greater interest and there are numerous papers on related subconvexity problems.
Recently, the first author and Putrius \cite{garPut} extended the results of \cite{gonek} in the setting of $L$-functions from the Selberg class.

The first to consider zeta-functions built from zeros was Mellin \cite{mellin}; this topic had several renaissances, first by Delsarte \cite{delsarte} (as mentioned above), Chakravarty \cite{chakra0,chakra}, and Levinson \cite{levinson} around 1970; then later by Jorgenson \& Lang \cite{jl} as well as Voros \cite{voros}.

	\bigskip

	\noindent
	Ram\={u}nas Garunk\v{s}tis\\
	Institute of Mathematics, Faculty of Mathematics and Informatics, Vilnius University, \\
    Naugarduko 24, LT-03225 Vilnius, Lithuania\\
    ramunas.garunkstis@mif.vu.lt
\bigskip

\noindent
Athanasios Sourmelidis\\
Univ. Lille, CNRS\\%
UMR 8524 - Laboratoire Paul Painlevé\\%
F-59000~Lille, France\\
athanasios.sourmelidis@univ-lille.fr
\bigskip

\noindent
J\"orn Steuding\\
Department of Mathematics, W\"urzburg University\\
Emil Fischer-Str. 40, 97\,074 W\"urzburg, Germany\\
steuding@mathematik.uni-wuerzburg.de
	
\end{document}